\documentclass[11pt,reqno]{amsart}

\usepackage{enumerate}
\usepackage{amsmath}%
\usepackage{amsfonts}%
\usepackage{amssymb}%
\usepackage{pgfplots}
\usepackage{graphicx}
\usepackage{mathrsfs}
\usepackage{hyperref}
\usepackage{datetime}
\usepackage{fullpage}
\usepackage{setspace}
\newtheorem{theorem}{Theorem}
\theoremstyle{plain}

\newtheorem{corollary}[theorem]{Corollary}

\newtheorem{definition}[theorem]{Definition}

\newtheorem{fact}[theorem]{Fact}
\newtheorem{lemma}[theorem]{Lemma}

\newtheorem{proposition}[theorem]{Proposition}
\newtheorem{question}[theorem]{Question}

\numberwithin{equation}{section}
\numberwithin{theorem}{section}
\numberwithin{case}{section}

\numberwithin{subcase}{case}
\usepackage{lineno}
\newcommand*\patchAmsMathEnvironmentForLineno[1]{%
\expandafter\let\csname old#1\expandafter\endcsname\csname #1\endcsname
\expandafter\let\csname oldend#1\expandafter\endcsname\csname end#1\endcsname
\renewenvironment{#1}%
{\linenomath\csname old#1\endcsname}%
{\csname oldend#1\endcsname\endlinenomath}}% 
\newcommand*\patchBothAmsMathEnvironmentsForLineno[1]{%
\patchAmsMathEnvironmentForLineno{#1}%
\patchAmsMathEnvironmentForLineno{#1*}}%
\AtBeginDocument{%
\patchBothAmsMathEnvironmentsForLineno{equation}%
\patchBothAmsMathEnvironmentsForLineno{align}%
\patchBothAmsMathEnvironmentsForLineno{flalign}%
\patchBothAmsMathEnvironmentsForLineno{alignat}%
\patchBothAmsMathEnvironmentsForLineno{gather}%
\patchBothAmsMathEnvironmentsForLineno{multline}%
}
\def\A{\mathcal{A}}
\def\B{\mathcal{B}}
\def\C{\mathcal{C}}
\def\D{\mathcal{D}}

\def\F{\mathcal{F}}
\def\G{\mathcal{G}}
\def\h{\mathcal{H}}
\def\J{\mathcal{J}}

\begin{document}
%\linenumbers \onehalfspace
\onehalfspace
\title[Beyond the Hilton--Milner Theorem]{The maximum size of a non-trivial intersecting uniform family that is not a subfamily of the Hilton--Milner family}

\author{Jie Han}
\author{Yoshiharu Kohayakawa}

\address{Instituto de Matem\'{a}tica e Estat\'{\i}stica, Universidade de S\~{a}o Paulo, Rua do Mat\~{a}o 1010, 05508-090, S\~{a}o Paulo, Brazil}

\email{jhan@ime.usp.br, yoshi@ime.usp.br}

\thanks{The first author is supported by FAPESP (2014/18641-5, 2015/07869-8).
  The second author is partially supported by FAPESP (2013/03447-6,
  2013/07699-0), CNPq (459335/2014-6, 310974/2013-5 and 477203/2012-4)
  and the NSF (DMS~1102086).  The authors acknowledge the support of
  NUMEC/USP (Project MaCLinC/USP)}

%\email[Yoshiharu Kohayakawa]{yoshi@ime.usp.br}
%\email[Jie Han]{jhan@ime.usp.br}

\shortdate
\yyyymmdddate
\settimeformat{ampmtime}
%\date{\today}
%\date{\today, \currenttime}


\begin{abstract}
  The celebrated Erd\H{o}s--Ko--Rado theorem determines the maximum
  size of a $k$-uniform intersecting family.  The Hilton--Milner
  theorem determines the maximum size of a $k$-uniform intersecting
  family that is not a subfamily of the so-called Erd\H{o}s--Ko--Rado
  family.  In turn, it is natural to ask what the maximum size of an
  intersecting $k$-uniform family that is neither a subfamily of the
  Erd\H{o}s--Ko--Rado family nor of the Hilton--Milner family is.  For
  $k\ge 4$, this was solved (implicitly) in the same paper by
  Hilton--Milner in 1967.  We give a different and simpler proof,
  based on the shifting method, which allows us to solve all cases
  $k\ge 3$ and characterize all extremal families achieving the
  extremal value.
\end{abstract}

\maketitle

\section{Introduction}
Let $X$ be an $n$-element set.  For $1\le k\le n$, let $\binom{X}k$ denote the family of all subsets of $X$ of cardinality $k$. 
A family $\F\subseteq 2^X$ (or a hypergraph) is called \emph{intersecting} if for all $F_1, F_2\in \F$, we have $F_1\cap F_2\neq \emptyset$.
A family $\F$ is \emph{$k$-uniform} if every member of $\F$ contains exactly $k$ elements.
An intersecting family $\F$ is \emph{trivial} if $\bigcap_{F\in \F} F\neq \emptyset$, i.e., there is an element that is common to all members of $\F$.
For families $\h \subseteq 2^X$ and $\G \subseteq 2^Y$, we write $\h\subseteq_S \G$ if there is an injective map $f:X \rightarrow Y$ such that $f(E)\in \G$ for every $E\in \h$.
If both $\h\subseteq_S \G$ and $\G\subseteq_S \h$ hold, then we say
that $\h$ and~$\G$ are isomorphic.  For simplicity, we often abuse notation
and write $\h=\G$ if~$\h$ and~$\G$ are isomorphic.

The celebrated Erd\H{o}s--Ko--Rado theorem~\cite{EKR} determines the
maximum size of a uniform intersecting family.  For any~$x\in X$, let
$\F(x)$ be the intersecting $k$-uniform family
$\{F\in \binom Xk : x\in F\}$.  We call~$x$ the \emph{center} of $\F(x)$.  We write~$\F_0$ for any family
isomorphic to~$\F(x)$.

\begin{theorem}[The Erd\H os--Ko--Rado theorem~\cite{EKR}]
  \label{thm:EKR}
  Let $\F$ be a $k$-uniform intersecting family on $X$ and
  suppose~$n\ge 2k$.  Then $|\F|\le \binom{n-1}{k-1}$. If $n > 2k$,
  equality holds only for~$\F_0$.
\end{theorem}

For non-trivial intersecting families, Hilton and Milner~\cite{HM67}
proved Theorem~\ref{thm:HM} below, whose statement is simplified if we
introduce some notation first.  For any $k$-set~$F\subset X$ and any
$x\in X\setminus F$, let
$\F(F,x)=\{ F\}\cup \{G\in \binom{X}k:x\in G, \,F\cap
G\neq\emptyset\}$.  We call~$x$ the \emph{center} of $\F(F, x)$.
We write~$\F_1$ for any family isomorphic to such a family~$\F(F,x)$.
We also define the following families: for any $3$-set~$S\subset X$,
let $\mathcal{T}(S):=\{F\in \binom{X}{k}: |F\cap S|\ge 2\}$.  We
write~$\G_2$ for any family isomorphic to such a
family~$\mathcal{T}(S)$.  The celebrated \textit{Hilton--Milner
  theorem} is as follows (for alternative proofs and generalizations,
see, e.g., Borg~\cite{borg13:_non}, Frankl and F\"uredi~\cite{FF86}
and Frankl and Tokushige~\cite{frankl92:_some}).

\begin{theorem}[The Hilton--Milner theorem~\cite{HM67}]
  \label{thm:HM}
  Let $\F$ be a non-trivial $k$-uniform intersecting family on $X$
  with $k\ge 2$ and $n>2k$. Then
  $|\F|\le \binom{n-1}{k-1} - \binom{n-k-1}{k-1} +1$. Equality holds
  only for the family~$\F_1$ and, if~$k\in\{2,3\}$, for the
  family~$\G_2$.
\end{theorem}

The Hilton--Milner theorem determines the maximum size of an
intersecting family \emph{that is not EKR}, i.e., that is not
contained in an~$\F_0$.  Note that the bound in the Hilton--Milner
theorem is much smaller than the bound in the Erd\H{o}s--Ko--Rado
theorem (as long as~$k$ is not too large) and, therefore, the
Hilton--Milner theorem shows the so-called ``stability'' of the
Erd\H{o}s--Ko--Rado theorem in a very strong sense (for other
stability-type results related to the Erd\H{o}s--Ko--Rado theorem, we
refer the reader to Dinur and Friedgut~\cite{dinur09:_inter},
Keevash~\cite{keevash08:_shadow} and Keevash and
Mubayi~\cite{keevash10:_set}).  What lies beyond the 1967 theorem of
Hilton and Milner, that is, beyond Theorem~\ref{thm:HM}?  Let us say
that a family~$\F$ \textit{is HM} if it is contained in the
family~$\F_1$ or~$k\in\{2,3\}$ and it is contained in~$\G_2$.  Our
question is then the following:

\begin{question}\label{q1}
  What is the maximum size of an intersecting family $\h$ that is
  neither EKR nor HM?  Which families achieve the extremal
  value?$\,$\footnote{This question has also been asked on MathOverflow
    \url{http://mathoverflow.net/q/94438}.}
\end{question} 

In fact, this question was partially answered in the 1967 paper of
Hilton and Milner~\cite{HM67}.  Their main result in that
paper~\cite[Theorem 3]{HM67}, which contains Theorem~\ref{thm:HM}
above, is as follows.  Here, for simplicity, we state it only for
$k$-uniform families.

\begin{theorem} \cite{HM67} \label{thm:HM2}
Fix integers $\min\{3, s\}\le k\le n/2$ and let $\F=\{A_1,\dots, A_m\}$ be a $k$-uniform intersecting family on $X$.
Moreover, assume that for any $S\subseteq[m]$ with $|S| > m - s$, we have
\[
\bigcap_{i\in S} A_i = \emptyset.
\]
Then 
\begin{equation}
  \label{eq:HM2}
  m \le \begin{cases}
    \binom{n-1}{k-1} - \binom{n-k}{k-1} + n - k &\text{ if } 2 < k\le s+2, \\
    \binom{n-1}{k-1} - \binom{n-k}{k-1} + \binom{n-k-s}{k-s-1} + s
    &\text{ if } k\le 2 \text{ or } k\ge s+2. 
  \end{cases}
\end{equation}
Moreover, the bounds in~\eqref{eq:HM2} are best possible.
\end{theorem}

Theorem~\ref{thm:HM2} contains the Erd\H{o}s--Ko--Rado theorem as its
special case~$s=0$, and it contains the Hilton--Milner theorem as its
special case~$s=1$.  Let us now consider Question~\ref{q1}.  Suppose
that~$k\geq4$ and~$\h$ is neither EKR nor HM.  Then~$\h$ satisfies the
hypothesis of Theorem~\ref{thm:HM2} for~$s=2$, and hence we know
that~$|\h|$ is at most as large as specified in the second bound
in~\eqref{eq:HM2}.  Unfortunately, when $k=3$, Theorem~\ref{thm:HM2}
does not give a sharp bound.  Moreover, Theorem~\ref{thm:HM2} does not
give any information about the extremal families that achieve the extremal
values.

In this note we settle Question~\ref{q1} completely.  Let us start by
noting that the case~$k=2$ is trivial.  Suppose that~$k=2$ and
that~$\h$ is not trivially intersecting.  Then~$\h$ is a triangle,
%(possibly plus some isolated vertices)
which means that $\h=\F_1$ (in
fact $\h=\F_1=\G_2$ when $k=2$).  This means that, for~$k=2$, every
intersecting family~$\h$ is either EKR or HM.  We may therefore
suppose that~$k\geq3$ in what follows.  Let us now describe the extremal
families for our theorem.

\begin{definition}[$\G(E,x_0)$, $\J(E,J,x_0)$, $\G_i$, $\J_i$]
  \label{ex:1}
  % Let $\G(E,x_0)$ and $\J(E,J,x_0)$ be the $k$-uniform families on $X$ defined as
  % follows.  
  For any $i$-set $E\subseteq X$, where $2\le i\le k$, and any
  $x_0\in X\setminus E$, we define the $k$-uniform family~$\G(E,x_0)$
  on~$X$ by setting
  \[
    \G(E,x_0) = \big\{G\in{\textstyle{X\choose k}}: E\subseteq G\}
    \cup \{G\in{\textstyle{X\choose k}}: x_0\in G, \,G\cap E \neq
  \emptyset\big\}. 
  \]
  We write~$\G_i$ for any family isomorphic to a~$\G(E,x_0)$ as above.
  Now suppose $1\le i\le k-1$.  For any $(k-1)$-set $E\subseteq X$,
  any $(i+1)$-set $J\subseteq X\setminus E$ and any $x_0\in J$, 
  we define the $k$-uniform family~$\J(E,J,x_0)$ on~$X$ by setting
  \begin{multline*}
  \qquad\J(E,J,x_0)=\big\{G\in{\textstyle{X\choose k}}: E\subseteq G, \,G\cap
  J\neq\emptyset\big\}
  \cup\big\{G\in{\textstyle{X\choose k}}: J\subseteq G\big\}\\
  \cup\big\{G\in{\textstyle{X\choose k}}:x_0\in
  G,\,G\cap E\neq\emptyset\big\}.\qquad
  \end{multline*}
  We write~$\J_i$ for any family isomorphic to a~$\J(E,J,x_0)$ as above.
\end{definition}

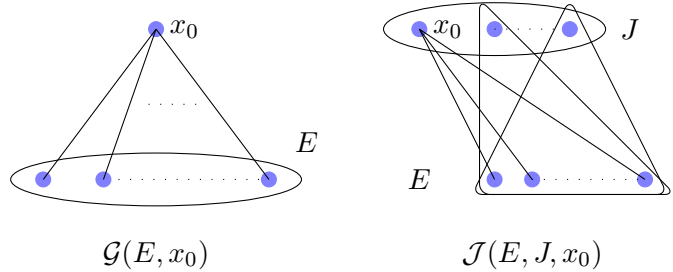
\begin{figure}[h]
\begin{center}
\begin{tikzpicture}
[inner sep=2pt,
   vertex/.style={circle, draw=blue!50, fill=blue!50},
   ]
\draw (0,0) ellipse (55pt and 10pt);
%\draw (5.5,0) ellipse (42pt and 10pt);
\node at (0,2) [vertex, label=right:$x_0$] {};
%\node at (0,1) [vertex, label=left:$y_1$] {};
\node at (-1.5,0) [vertex]{};
%\node at (-1,0) [vertex]{};
\node at (-0.7,0) [vertex]{};
\node at (1.5,0) [vertex]{};
\node at (2,0.5) {$E$};
\draw[loosely dotted] (-0.5,0) -- (1.5,0);
\draw[loosely dotted] (-0.1,1) -- (0.6,1);
\node at (0,-1) {$\G(E,x_0)$};

%\node at (5,1) [vertex, label=right:$y_2$] {};
\node at (3.5,2) [vertex, label=right:$x_0$] {};
\node at (4.5,0) [vertex]{};
\node at (5,0) [vertex]{};
\node at (6.5,0) [vertex]{};
%\node at (5,0) [vertex]{};

\node at (3.5,0) {$E$};
\node at (6.3,2) {$J$};

\node at (4.5,2) [vertex]{};
\draw[loosely dotted] (4.5,2) -- (5.5,2);
\node at (5.5,2) [vertex]{};

\draw[loosely dotted] (5,0) -- (6.5,0);
%\draw[loosely dotted] (4.9,1) -- (5.6,1);
\node at (5,-1) {$\J(E,J,x_0)$};
\draw (0,2) -- (-1.5,0);
%\draw (0,2) -- (-1,0);
\draw (0,2) -- (-0.7,0);
\draw (0,2) -- (1.5,0);
\draw (3.5,2) -- (4.5,0);
\draw (3.5,2) -- (5,0);
\draw (3.5,2) -- (6.5,0);
%\line
\draw [rounded corners] (4.3,2.4) -- (4.3,-0.2) -- (6.9,-0.2) -- cycle;
\draw [rounded corners] (5.5,2.4) -- (4.2,-0.2) -- (6.8,-0.2) -- cycle;
%\draw (5,2) -- (6.5,0);
\draw (4.5,2) ellipse [x radius=42pt, y radius=10pt];
\end{tikzpicture}

\caption{Extremal families: $\G(E,x_0)$ and $\J(E,J,x_0)$. Each family consists of all $k$-sets that contain some pair (line segment) or set (ellipse or triangle) in the picture.}
\end{center}
\end{figure}

Note that $\J_1 = \F_1 = \G_k$, the Hilton--Milner family.  We also
remark that the families~$\G_i$ above appear as the extremal families
in a result of Frankl~\cite{frankl87:_erd_ko_rado} that generalizes
the Hilton--Milner theorem.  Our main result is as follows.

\begin{theorem}\label{thm:BeyHM}
  Suppose $k\ge3$ and $n>2k$ and let $\h$ be an intersecting
  $k$-uniform family on $X$. Furthermore, assume that
  $\h\nsubseteq_S \F_0$, $\h\nsubseteq_S \F_1$ and, if $k=3$,
  $\h\nsubseteq_S \G_2$.  Then
  \begin{equation}
    \label{eq:BeyHM}
    |\h|\le \binom{n-1}{k-1} - \binom{n-k-1}{k-1} - \binom{n - k -
    2}{k-2} + 2.
  \end{equation}
  For $k=4$, equality holds if and only if $\h = \J_2$, $\G_{2}$ or
  $\G_{3}$; for every other~$k$, equality holds if and only if $\h = \J_2$.
\end{theorem}

We now make some remarks on Theorem~\ref{thm:BeyHM}.  Let us first of
all mention that a quick calculation shows that the right-hand side
of~\eqref{eq:BeyHM} is equal to the second bound in~\eqref{eq:HM2}
with~$s=2$.  We now note that Theorem~\ref{thm:BeyHM} implies an
Erd\H os--Ko--Rado type theorem with a maximum degree condition.  For any
family~$\h$ on~$X$ and $x\in X$, let $d_\h(x)$ be the number of sets
in $\h$ containing~$x$. Let $d(\h)=\max_{x\in X}d(x)$. In
\cite{frankl87:_erd_ko_rado}, Frankl showed that, for any $k$-uniform
intersecting family $\h$ on~$[n]$ and $2\le i\le k$, if
$d(\h) \le d(\G_i)$, then $|\h| \le |\G_i|$.  When $k\ge 4$, this
generalizes the Hilton--Milner theorem because the case $i=k$ is
equivalent to the Hilton--Milner theorem.  Theorem~\ref{thm:BeyHM}
also gives such a maximum degree form for $k\ge 4$.

\begin{corollary}\label{thm:MaxDeg}
  Suppose~$k\ge 4$ and let~$\h$ be a $k$-uniform intersecting family 
  on~$X$.  If $d(\h) \le d(\J_2)$, then $|\h| \le |\J_2|$.  Moreover,
  if~$k\ge 5$, equality holds if and only if~$\h = \J_2$.
\end{corollary}

\begin{proof}
If $\h$ satisfies the assumptions of Theorem~\ref{thm:BeyHM}, then $|\h| \le |\J_2|$. If $\h\subseteq_S\F_0$ or $\h\subseteq_S\F_1$, then we have $|\h| \le d(\h) +1 \le d(\J_2) +1 < |\J_2|$.
Clearly, if~$k\geq5$ equality holds only if~$\h$ satisfies the
assumptions of Theorem~\ref{thm:BeyHM} and, therefore, $\h=\J_2$.
\end{proof}

For another Erd\H{o}s--Ko--Rado type theorem with conditions on the
maximum degree, see F\"u\-re\-di~\cite{Furedi78}.  We also remark that
there is another way of ``going beyond'' the Hilton--Milner theorem,
namely, Frankl~\cite{Frankl80} investigated the maximum size of a
$k$-uniform intersecting family whose transversal number (the minimum
size of a vertex cover) is more than~$2$.  We close by observing that
``our way'' of going beyond Hilton--Milner has very recently been
considered by Jackowska, Polcyn and Ruci\'{n}ski~\cite{JPR15}, in a
more general form (those authors define a certain hierarchy of Tur\'an
numbers; see~\cite{JPR15}).

\section{Proof of the quantitative part of Theorem~\ref{thm:BeyHM}}

We use the \emph{shifting technique} in~\cite{EKR, FF86}.  Readers not
familiar with shifting are strongly encouraged to study~\cite{frankl87:_shifting}.
Here we give the definition of the shifting operator and briefly state some basic facts. 
To define shifting, we need to suppose that the elements of~$X$ are
given some linear order. 
For $x, y\in X$, $x<y$, we define $S_{xy}(\h) = \{S_{xy}(E) : E\in \h \}$, where
\[
S_{xy} (E) = \begin{cases}
(E \setminus \{y\})\cup \{x\} & \text{ if } x\notin E, y\in E, (E \setminus \{y\})\cup \{x\} \notin \h, \\
E  &\text{ otherwise.}
\end{cases}
\]

\begin{proposition}\cite{EKR}
  We have $|S_{xy}(\h)| = |\h|$.  Moreover, $S_{xy}(\h)$ is
  intersecting if $\h$ is intersecting.
\end{proposition}

Let $\h$ be an intersecting $k$-uniform family on $X=[n]$.
To prove Theorem~\ref{thm:EKR}, we apply the shifting operator
$S_{xy}$ repeatedly to $\h$ for all $1\le x<y\le n$ until we get a
\emph{stable} family $\G$, i.e., such that $S_{xy}(\G) = \G$ holds for all $1\le x<y\le n$.
We note that the shifting process must terminate and thus we always reach a stable family.
Indeed, note that for any shift $S_{xy}$ on $\h$, if $S_{xy}(E)\neq E$, then the sum of the elements in $S_{xy}(E)$ is strictly smaller than the sum of the elements in $E$ (because we replaced $y$ by $x$ and $x<y$).
Thus the sum, over all edges of~$\h$, of all such sums strictly decreases unless $S_{xy}(\h)=\h$.
So the shifting process must terminate.

We say a family $\h$ is \emph{EKR $($or HM\,$)$ at $x\in X$} if $\h\subseteq_S \F_0$ (or $\h\subseteq_S \F_1$) where $x$ is mapped to the center of $\F_0$ (or $\F_1$).
We say a family $\h$ is HM at $\{x, y, z\}\subseteq X$ if
$\h\subseteq_S \G_2$ where $\{x, y, z\}$ is mapped onto the set $\{x\}\cup E$ of $\G_2$ in Definition~\ref{ex:1}.

Throughout this section, we will use the following fact, whose proof is trivial.

\begin{fact}\label{fact:HM}
Let $\h$ be a $k$-uniform intersecting family on $X$ and let $x\in X$. Then $\h$ is neither EKR nor HM at $x$ if and only if there are $E$, $E'\in \h$ such that $x\notin E$, $x\notin E'$.
\end{fact}

Let $\h$ be an intersecting family of maximal size such that
$\h\nsubseteq_S \F_0$, $\h\nsubseteq_S \F_1$ and if $k=3$,
$\h\nsubseteq_S \G_2$.  We prove the statement by
induction on $k\ge 2$.  The base case $k=2$ is trivial, because there
is no such $\h$.  For the case $k=3$ and $n=2k+1=7$, note that the HM
family has size 13, which implies $|\h|\le 12$, as desired (we will
show the uniqueness of this case in Section 3.3). Hence, throughout this section, when $k=3$, we may assume that $n \ge 2k+2=8$.

We plan to apply repeatedly the shifting operator $S_{xy}$ to $\h$ for all $1\le x<y\le n$. But we may fall into trouble if the family after shifting becomes a subfamily of $\F_0$ or $\F_1$ (or $\G_2$ if $k=3$). We observe the following facts.

\begin{fact}\label{fact:shift}
For any $\h'\nsubseteq_S \F_0$, $\h'\nsubseteq_S \F_1$ and if $k=3$, $\h'\nsubseteq_S \G_2$, we have the following.
\begin{itemize}
\item[$(i)$] If $S_{xy}(\h')$ is EKR (or HM) at some element, then $S_{xy}(\h')$ is EKR (or HM) at $x$.
\item[$(ii)$] If $S_{xy}(\h')$ is HM at a 3-set $A$, then $A=\{x,x_1,x_2\}$ for some $x_1, x_2\in X\setminus\{x, y\}$.
\end{itemize}
\end{fact}

\begin{proof}
For any $z\in X\setminus \{x\}$, by Fact~\ref{fact:HM}, there are at least two edges in $\h'$ that do not contain $z$. 
Clearly, after the shift $S_{xy}$, these edges do not contain $z$ as well. So $S_{xy}(\h')$ is neither EKR nor HM at $z$ and $(i)$ follows.

For $(ii)$, since $\h'$ is not HM at any 3-set $\{x_0, x_1, x_2\}$, there exists an edge $E\in \h'$ such that $|E\cap \{x_0, x_1, x_2\}|\le 1$. 
If $S_{xy}(\h')$ is HM at $\{x_0, x_1, x_2\}$, then $|S_{xy}(E)\cap \{x_0, x_1, x_2\}|\ge 2$. This happens only if one of $\{x_0, x_1, x_2\}$ is $x$ and none of them is $y$. We may assume that $x_0=x$ and thus $(ii)$ follows.
\end{proof}

Note that $\G_2$ is the extremal family in Theorem~\ref{thm:HM} only for $k=3$. So when $k\ge 4$, if we get a family $\h'$ such that $S_{xy}(\h')$ is HM at a 3-set, then we continue the shifting process.
By Fact~\ref{fact:shift}, if we apply~$S_{xy}$ ($x<y$) repeatedly to~$\h$, then we obtain a family in one of the following four cases,
\begin{itemize}
\item[(0)] a family $\G$ which is stable, i.e., $S_{xy}(\G) = \G$ holds for all $x<y$,
\item[(1)] a family $\h_1$ such that $S_{xy}(\h_1)$ is EKR at $x$,
\item[(2)] a family $\h_2$ such that $S_{xy}(\h_2)$ is HM at $x$, or
\item[(3)] (for $k=3$ only) a family $\h_3$ such that $S_{xy}(\h_3)$ is HM at $\{x, x_1, x_2\}$ for some $x_1, x_2\in X\setminus \{x, y\}$.
\end{itemize}

In Cases (1) -- (3), we will \textit{not} apply the shift $S_{xy}$ --
otherwise we will get a family whose size is out of our control.
Instead, we will adjust our shifting as shown in the following
proposition.

\begin{proposition}\label{prop:adjust}
In Case $(i)$, for $i=1,2,3$, there is a set $X_i\subset X$ of size at most 4 such that for all $x', y'\in X \setminus X_i$, $x'<y'$, we can apply (repeatedly) all the shifts $S_{x'y'}$, i.e., we will not be in any of Cases (1) -- (3) and the resulting family $\G$ satisfies that $S_{x'y'}(\G)=\G$ for all $x', y'\in X \setminus X_i$, $x'<y'$. Moreover, $E\cap X_i \neq \emptyset$ for all $E\in \G$ and when $k\ge 4$, the sets $X_i$ can be chosen of size at most 3.
\end{proposition}

\begin{proof}
We first assume $k\ge 4$. In Case (1), define $X_1 = \{x, y\}$ and note that for any $E\in \h_1$, $E\cap X_1\neq \emptyset$.
Indeed, for any $E\in \h_1$, since $S_{xy}(E)$ contains $x$, we know that if $x\notin E$, then $y\in E$.
In Case (2), observe that $S_{xy}(\h_2)$ contains exactly one edge $E_0=\{z_1, \dots, z_k\}$ that does not contain~$x$.
Without loss of generality, assume that $z_1\neq y$ and let $X_2 = \{x, y, z_1\}$. 
Also note that for any $E\in \h_2\setminus \{E_0\}$, we have $E\cap\{x, y\}\neq \emptyset$. Thus, for any $E\in \h_2$, $E\cap X_2 \neq \emptyset$.

For $i=1,2$, apply repeatedly $S_{x'y'}$ to the family for $x'<y'$, $x', y'\in X \setminus X_i$ and we claim that we will reach a family $\G$ such that $S_{x'y'}(\G)=\G$ for any $x', y'\in X \setminus X_i$, $x'<y'$.
Indeed, by Fact~\ref{fact:HM}, it suffices to show that in each step, the current family $\h'$ contains at least two $k$-sets that do not contain $x'$. 
Since for any $E\in \h_i$, $E\cap X_i \neq \emptyset$, the maximality of $|\h|$ implies that all $k$-sets containing $X_i$ are in $\h$ for $i=1,2$. Moreover, these sets stay fixed during the shifting process. 
This implies that there are at least $\binom{n-3}{k-2}$ (if $i=1$) or $\binom{n-4}{k-3}$ (if $i=2$) members of $\h'$ that do not contain $x'$ and we are done. 

Now we prove the case $k=3$. The following observation will be helpful.
\begin{fact}\label{fact:33}
Given a 3-uniform family $\F$ on $X$ which is HM at a 3-set $A$, if there are at least three triples of $\F$ containing both $v, v'\in X$, then $\{v, v'\}\subseteq A$.
\end{fact}

First assume that we reach Case (1).
Let $X_1 = \{x, y\}$ and note that for any $E\in \h_1$, $E\cap X_1 \neq \emptyset$.
We apply~$S_{x'y'}$ repeatedly to the family for $x'<y'$, $x', y'\in X \setminus X_1$. 
By Fact~\ref{fact:HM}, to show that we will \emph{not} reach Cases (1) or (2) for any $x'<y'$, $x', y'\in X \setminus X_1$, it suffices to show that in each step, the current family $\h'$ contains at least two triples that do not contain $x'$. 
Since for any $E\in \h_1$, $E\cap X_1 \neq \emptyset$, the maximality of $|\h|$ implies that all the $n-2$ triples that contain $X_1$ are in $\h$ (so in $\h'$).
Moreover, these triples stay fixed during the shifting process. 
We are done because there are $n-3$ such triples in $\h'$ that do not contain $x'$. 
However, it is possible that we reach Case (3) this time.
Assume that we reach Case (3), say, $\h' := S_{x' y'}(\h'')$ is HM at some 3-set $A$ for some $x' < y'$.
We claim that $A=\{x', x, y\}$.
Indeed, since $\{x,y\}=X_1$ is in at least $n-2\ge 6$ edges of $\h'$, we know that $\{x, y\}\subseteq A$ by Fact~\ref{fact:33}.
Moreover, by Fact~\ref{fact:shift}, we have $x'\in A$ and thus $A=\{x', x, y\}$.
Since $\h''$ is not HM at $\{x', x, y\}$, we can pick an element $z'$ such that $\{x, y', z'\}\in \h''$ or $\{y, y', z'\}\in \h''$.
We then set $X_1 = \{x, y, y', z'\}$ and do the shift for all $x'', y'' \in X \setminus X_1$, $x'' < y''$.
The same arguments show that this time we will not reach Case (1) or (2). 
By similar reasons as before, the resulting family can only be HM at $\{x'', x, y\}$.
This is also impossible because the family contains $\{x, y', z'\}$ (or $\{y, y', z'\}$).

Second assume that we reach Case (2).
Let $X_2 = \{x, y, z_1, z_2\}$, where $E_0=\{z_1, z_2, z_3\}$ is defined as in the case $k\ge 4$ and without loss of generality, $z_1\neq y$ and $z_2\neq y$. 
Note that for any $E\in \h_2\setminus\{E_0\}$, $E\cap\{x, y\}\neq \emptyset$. So $E\cap X_2 \neq \emptyset$ for all $E\in \h_2$, and moreover, by the maximality of $|\h|$, we may assume that $\{x, y, z_1\}, \{x, y, z_2\}\in \h_2$.
We apply repeatedly $S_{x'y'}$ to the family for $x'<y'$, $x', y'\in X \setminus X_2$. 
In each step $\h'$, we know that $\{x, y, z_1\}, \{x, y, z_2\}\in \h'$, so we will not reach Case (1) or (2).
Moreover, if $\h'=S_{x'y'}(\h'')$ is HM at a 3-set $A$, then both $\{x, y, z_1\}$ and $\{x, y, z_2\}$ can miss at most one element of $A$. 
By Fact~\ref{fact:shift}, $x'\in A$ and thus $A = \{x', x, y\}$.
Recall that $E_0 = \{z_1, z_2, z_3\}\in \h_2$ and $\{z_1, z_2\}\cap \{x', x, y\}=\emptyset$. This implies that $|E_0'\cap \{x', x, y\}|\le 1$, where $E_0'\in \h'$ represents the set obtained from $E_0$ after a series of shifts. This is a contradiction and thus we will not reach Case (3). 

At last, assume that we reach Case (3).
Let $X_3 = \{x, y, x_1, x_2\}$.
Note that for any $E\in \h_3$, $|E\cap X_3| \ge 2$.
By the maximality of $|\h|$, we may assume that $\{x, y, x_1\}, \{x, y, x_2\}\in \h_3$.
We apply repeatedly $S_{x'y'}$ to the family for $x'<y'$, $x', y'\in X \setminus X_3$. 
In each step $\h'$, we know that $\{x, y, x_1\}, \{x, y, x_2\}\in \h'$, so we will not reach Case (1) or (2).
Moreover, assume that $\h'=S_{x'y'}(\h'')$ is HM at a 3-set $A$.
Since every set in $\h_3$ contains $x_1$ or $x_2$, by the maximality of $|\h|$, we may assume that all sets containing $x_1$ and $x_2$ are in $\h_3$ (so in $\h'$).
This implies that $\{x_1, x_2\}\subseteq A$ by Fact~\ref{fact:33}.
By Fact~\ref{fact:shift}, $x'\in A$ and thus $A = \{x', x_1, x_2\}$.
However, since $\{x, y, x_1\}\in \h'$ and $|\{x, y, x_1\}\cap A|=1$, we get a contradiction. So we will not reach Case (3).

Let $\G$ be the resulting family.
Note that $E\cap X_i \neq \emptyset$ for all $E\in \h_i$ implies that $E\cap X_i \neq \emptyset$ for all $E\in \G$, because the shifts do not affect elements in $X_i$.
\end{proof}

Eventually we obtain a family $\G$ such that
\begin{enumerate}[(i)]
\item $E\cap X_i \neq \emptyset$ for all $E\in \G$ and $i=1,2,3$,
\item $S_{x'y'}(\G) = \G$ for $x'<y'$, $x', y'\in X \setminus X_i$.
\end{enumerate}

Let $X_0 = \emptyset$. For $k\ge 4$ and $i=0,1,2$, let $Y_i$ be the set of the first $2k - |X_i|$ elements of $X \setminus X_i$. For $k=3$ and $i=0,1,2,3$, let $Y_i$ be the set of the first $7 - |X_i|$ elements of $X \setminus X_i$.
By Proposition~\ref{prop:adjust}, in all cases, we have $|Y_i | \ge 2k -3$.
If we end up with Case ($j$) in Proposition~\ref{prop:adjust} ($j\in\{1,2,3\}$), let $Y = X_j\cup Y_j$, and thus $|Y| = 2k$ or $2k+1$.

\begin{lemma}\label{lem:2kint}
For all $E, E'\in \G$, $E\cap E'\cap Y\neq\emptyset$ holds.
\end{lemma}

\begin{proof}
%The cases when $Y=X_0\cup Y_0$ and $Y=X_1\cup Y_1$ are shown in \cite{FF86}.
First let $i=1,2,3$.
Suppose for a contradiction that $E\cap E'\cap Y = \emptyset$ and $E, E'\in \G$ such that $|E\cap E'|$ is minimal. By (i) and $|E\cap E' \cap (X \setminus Y)|\ge 1$, we have
\[
|(E\cup E')\cap Y_i| \le |E\cap Y_i| + |E'\cap Y_i| \le 2k - 4.
\]
Since $|Y_i| \ge 2k -3$, there is an element $a\in Y_i \setminus (E\cup E')$. Pick any $b\in E\cap E'\cap (X \setminus Y)$ and note that $a<b$. By (ii), we know $E'' := (E' \setminus \{b\})\cup \{a\} \in \G$. This is a contradiction because $E\cap E''\cap Y=\emptyset$ and $|E\cap E''|< |E\cap E'|$.
The case $Y = X_0\cup Y_0$ is similar.
\end{proof}

For $i\in [k]$, let $\A_i = \{E\cap Y: E\in \G, |E\cap Y| = i\}$. By the definition of $\A_i$ and Lemma~\ref{lem:2kint}, we have the following fact.

\begin{fact}\label{fact:Ai}
The family $(\bigcup_{1\le i\le k}\A_i)\cup \G$ is intersecting.
\end{fact}

The following lemma is devoted to our final counting.

\begin{lemma}\label{lem:calc}
For $k=3$, we have $|\A_1|=0$, $|\A_2| \le 2$ and $|\A_3| \le 12$.
For $k\ge 4$, we have
\[
|\A_i| \le \binom{2k-1}{i-1} - \binom{k-1}{i-1} - \binom{k-2}{i-2} \text{ for } 1\le i\le k-1
\]
and
\[
|\A_k| \le \frac12\binom{2k}{k} = \binom{2k-1}{k-1} - \binom{k-1}{k-1} - \binom{k-2}{k-2} +2.
\]
\end{lemma}

\begin{proof}
If $|\A_1| > 0$, then there is a set $E\in \G$, such that $E\cap Y = \{x\}$ for some $x\in Y$. By Fact~\ref{fact:HM}, there is a set $E'\in \G$ such that $x\notin E'$. Thus $E\cap E' \cap Y = \emptyset$, contradicting Lemma~\ref{lem:2kint}. So $|\A_1| =0$ as desired.

First assume $k=3$. 
Assume to the contrary that~$|\A_2|\ge 3$. Since $\A_2$ is 2-uniform and intersecting, $\A_2$ is a star or a triangle.
If $\A_2$ is a star at~$x$ of size at least 3, then there is at most one triple that avoids $x$ and meets each edge of the star, contradicting Fact~\ref{fact:HM}.
Otherwise $\A_2$ is a triangle at $\{x, y, z\}$. In this case any member of $\A_2\cup\A_3$ (and thus any member of $\G$) must contain at least two elements of $\{x, y, z\}$, which means that $\G\subseteq_S \G_2$, a contradiction.
Thus we get $|\A_2|\le 2$.
Note that $\A_3$ is the induced subfamily of $\G$ on $Y$, which is a 3-uniform family on 7 elements ($|Y|=2k+1=7$).
Assume to the contrary that $|\A_3|\ge 13$.
By Theorem~\ref{thm:HM}, we know that $\A_3$ is either EKR or HM.
First assume that $\A_3$ is EKR or HM at some $x\in Y$, i.e., $d_{\A_3}(x)\ge |\A_3|-1$.
Since $\G$ is neither EKR nor HM, there is an edge $E\in \G\setminus \A_3$ such that $x\notin E$.
Moreover, because $\A_1=\emptyset$, we know that $x\notin E\cap Y\in \A_2$.
Since $\A_2\cup\A_3$ is intersecting (Fact~\ref{fact:Ai}), every edge in $\A_3$ must intersect $E\cap Y$.
This implies that $d_{\A_3}(x)\le 9$, and thus $|\A_3|\le d_{\A_3}(x)+1\le 10$, a contradiction.
Second assume that $\A_3$ is HM at some $\{x, y, z\}\in Y$, so $|\A_3|=13$.
Similarly, since $\G$ is not HM, there is an edge $E\in \G\setminus \A_3$ such that $|E\cap \{x, y, z\}|\le 1$.
Without loss of generality, assume that $y, z\notin E$.
Let $z'\in Y\setminus \{x, y, z\}$ and note that $\{y, z, z'\}\in \G$.
Then $E\cap \{y, z, z'\}=\emptyset$, contradicting that $\G$ is intersecting.
So $|\A_3|\le 12$ holds.

Now assume $k\ge 4$. Fix $2\le i\le k-1$. Observe that
\[
\binom{2k-1-i}{i-1} - \binom{k-1}{i-1} = \binom{2k-2-i}{i-2} + \cdots + \binom{k-1}{i-2} \ge 2.
\]
Indeed, since there are $k-i$ binomial coefficients in the sum and each of them is at least 1, the inequality holds if $i\le k-2$. Otherwise $i=k-1$, then $\binom{2k-1-i}{i-1} - \binom{k-1}{i-1} = \binom{k-1}{2}\ge 2$ as $k\ge 4$. 
Assume that, to the contrary of the inequality in the lemma, we have
\begin{align*}
|\A_i|& > \binom{2k-1}{i-1} - \binom{k-1}{i-1} - \binom{k-2}{i-2}  \\
&\ge \binom{2k-1}{i-1} - \binom{2k-1-i}{i-1} - \binom{2k-2-i}{i-2}+2.
\end{align*}
Since $\A_i$ is an intersecting $i$-uniform family on $2k$ vertices,
we may assume, by induction on~$i$, 
that~$\A_i$ is EKR or HM at some $x\in Y$, or $\A_i$ is HM at some $\{x,y,z\}\subseteq Y$ for $i=3$.

We first assume that $\A_i$ is EKR or HM at some $x$. So $\A_i$ contains at most one $i$-set $A$ which does not contain $x$.
Pick $E, E'\in \G$ such that $x\notin E$, $x\notin E'$ and $|E\cap Y|$ is minimal.
Let $|E\cap E'\cap Y| = t$, $|(E\cap Y)\setminus E'|=t_1$ and $|(E'\cap Y)\setminus E|=t_2$. 
Clearly, $1\le t\le k-1$ and $t+t_1\le t+t_2\le k$.
Since $\A_i\cup \{E, E'\}$ is intersecting, we have
\begin{equation}\label{eq:Ai}
|\A_i| \le \binom{2k-1}{i-1} - \binom{2k-1-t-t_1}{i-1} - \binom{2k-1-t-t_2}{i-1} + \binom{2k-1-t-t_1-t_2}{i-1}+c,
\end{equation}
where $c=1$ if $\A_i$ contains an $i$-set that does not contain $x$ and $c=0$ otherwise.
Note that
\begin{align*}
&- \binom{2k-1-t-t_1}{i-1} - \binom{2k-1-t-t_2}{i-1} + \binom{2k-1-t-t_1-t_2}{i-1} \\
=&- \binom{2k-1-t-t_1}{i-1} - \binom{2k-2-t-t_2}{i-2} - \cdots - \binom{2k-1-t-t_1-t_2}{i-2}.
\end{align*}
%and the similar equation holds after exchanging the roles of $t_1$ and $t_2$.
So in \eqref{eq:Ai}, we can substitute $t_1$ and $t_2$ by $k-t$ (this will not decrease the bound), that is,
\[
|\A_i| \le \binom{2k-1}{i-1} - \binom{k-1}{i-1} - \binom{k-1}{i-1} + \binom{t-1}{i-1}+c.
\]
Similarly, we can substitute $t$ by $k-1$, that is, $|\A_i| \le \binom{2k-1}{i-1} - 2\binom{k-1}{i-1} + \binom{k-2}{i-1}+c$.
Moreover, the inequality is tight only if $t+t_1=k$, but $c=1$ holds only if $t+t_1\le i\le k-1$.
Since we cannot have both holding simultaneously, we have the desired bound
\begin{align*}
|\A_i| \le \binom{2k-1}{i-1} - 2 \binom{k-1}{i-1} + \binom{k-2}{i-1} = \binom{2k-1}{i-1} - \binom{k-1}{i-1} - \binom{k-2}{i-2}.
\end{align*}

Next assume that $i=3$ and $\A_i$ is HM at some $\{x,y,z\}\subseteq Y$. In this case it is easy to see that $|\A_i|\le 3(2k-3)+1 = 6k-8$. Note that when $k\ge 4$, we have $6k-8 \le \binom{2k-1}{i-1} - \binom{k-1}{i-1} - \binom{k-2}{i-2}$. %(the equality holds only if $k=4$).

At last, by Theorem~\ref{thm:EKR}, we have $|\A_k| \le \binom{2k-1}{k-1}= \binom{2k-1}{k-1} - \binom{k-1}{k-1} - \binom{k-2}{k-2} +2$.
\end{proof}

Now we proceed to the final estimation.
Note that for a fixed $A\in \A_i$, there are at most $\binom{n-|Y|}{k-i}$ $k$-element sets $E$ with $E\cap Y=A$. For $k=3$, we get
\begin{equation}\label{eq:eq0}
|\G| \le \sum_{i=1}^k |\A_i| \binom{n-2k-1}{k-i} \le 2\binom{n-2k-1}{k-2} + 12 = 2n-2,
\end{equation}
as desired. For $k\ge 4$, we have
\begin{align}
|\G| &\le \sum_{i=1}^k |\A_i| \binom{n-2k}{k-i} \le  2 + \sum_{i=1}^k \left(\binom{2k-1}{i-1} - \binom{k-1}{i-1} - \binom{k-2}{i-2} \right)\binom{n-2k}{k-i} \nonumber \\
&= \binom{n-1}{k-1} - \binom{n-k-1}{k-1} - \binom{n-k-2}{k-2}+2 \label{eq:eq},
\end{align}
proving the inequality part of the theorem.

\section{The uniqueness in the theorem}

\subsection{The stability of the shifts}
The following lemma shows the `stability' of the shifts.

\begin{lemma}\label{lem:shift}
Let $\h$ be a $k$-uniform intersecting family. If $k\ge 3$ and $S_{xy}(\h) = \F$ for some $\F\in \{\J_2, \G_{k-1}, \G_2\}$, then $\h$ is isomorphic to $\F$. 
\end{lemma}
We first prove the following propositions.
For two families $\A_1$ and~$\A_2$ on the same set, we say $(\A_1, \A_2)$ is \emph{cross-intersecting} if for any $A_1\in \A_1$ and $A_2\in \A_2$, $A_1\cap A_2\neq \emptyset$.
A family $\A$ is called \emph{non-separable} if for any partition $\A_1\cup \A_2$ of $\A$ such that $(\A_1, \A_2)$ is cross-intersecting, we have that $\A_1=\emptyset$ or $\A_2=\emptyset$.

\begin{proposition}\label{prop:1}
Fix $s\ge 2$ and~$a$ and~$b\ge 1$ such that $a+ b\le s$. Let $C$ be a set of size at least $s+1$ and let $z_1, z_2\notin C$. Let $\C$ be the family on $C\cup \{z_1, z_2\}$ such that
\[
\C = \{\{z_1\}\cup D: D\subseteq C, |D|=a\} \cup \{\{z_2\}\cup E: E\subseteq C, |E|=b\}.
\]
Then $\C$ is non-separable.
\end{proposition}

\begin{proof}
Consider a partition $\C=\C'\cup \C''$.
Since $|C|\ge s+1$, for any two $a$-sets $D_1, D_2$ in $C$ such that $|D_1\cap D_2| =a-1$, there is a $b$-set $E\subseteq C\setminus (D_1\cup D_2)$. 
Thus both $\{z_1\}\cup D_1$ and $\{z_1\}\cup D_2$ are disjoint from $\{z_2\}\cup E$ and thus they must belong to the same part. 
Observe that for any two $a$-sets $D, D'$ in $C$, we can pick a sequence of $a$-sets $D_1,\dots, D_t$ such that $|D_i\cap D_{i+1}|=a-1$ for $i\in [t-1]$ and $|D\cap D_1| = |D'\cap D_t|=a-1$. 
So all sets of form $\{z_1\}\cup D$ are in the same part. 
Clearly, for any $b$-set $E\subseteq C$, there exists $D\subseteq C$ such that $\{z_2\} \cup E$ and $\{z_1\}\cup D$ are disjoint and thus they must be in the same part. So all sets in $\C$ are in the same part and we are done.
\end{proof}

\begin{proposition}\label{prop:crossint}
Fix $r\ge 2$. Let $Z$ be a set of size $m\ge 2r+1$ and let $A\subseteq Z$ such that $|A|\in \{r-1,r\}$. Let $\B$ be an $r$-uniform family on $Z$ such that
$\B = \{B\subseteq Z: 0<|B\cap A| < |A| \}$. 
Then $\B$ is non-separable.
\end{proposition}

\begin{proof}
First we assume $r=2$. If $|A|=1$, then $\B=\emptyset$ and we are done. Otherwise $|A|=2$. Note that in this case~$\B$ is isomorphic to the complete bipartite graph $K_{2, m-2}$, where $m-2\ge 3$. Then the proposition follows immediately from Proposition~\ref{prop:1} by setting $a=b=1$, $s=2$ and $C=Z\setminus A$.

Now assume $r\ge 3$.
Consider a partition $\B=\B'\cup \B''$.
Let $\B_i = \{E\in \B: |E\cap A| = i\}$ for all $1\le i\le r-1$.

\medskip
\noindent\textbf{Claim. }\textit{$\B_1$ is non-separable.}
\medskip

Note that the claim implies the proposition.
Indeed, without loss of generality, assume that $\B_1\subseteq \B'$. Note that $\B_{r-1}=\emptyset$ when $|A|=r-1$. For any set $E$ in $\B_{r-1}$ or $\B_{r-2}$, we can find a set $E'$ in $\B_1$ that is disjoint from $E$. So $E$ and thus all sets in $\B_{r-1}$ and $\B_{r-2}$ must be in $\B'$.
Similarly, using the sets in $\B_{r-2}$, we conclude that all sets in $\B_{2}$ must be in $\B'$.
%, which, in turn, implies that all sets in $\B_{r-3}$ are in $\B'$.
So we get $\B = \B'$ after iteratively applying the same arguments.

So it remains to prove the claim.

\medskip
\noindent\textit{Proof of the claim.}
Let $C=Z\setminus A$.
We first prove the case when $|A|=r$, say, $A=\{x_1,\dots, x_r\}$. Note that $|C| = m-r \ge r+1$.
Let $a=1$, $b=r-1$ and $s=r$. We replace $\{x_1,\dots, x_{r-1}\}$ by $z_1$ and set $z_2=x_r$. Applying Proposition~\ref{prop:1} shows that all sets in $\B_{r-1}$ that contain $\{x_1,\dots, x_{r-1}\}$ and all sets in $\B_1$ that contain $x_r$ are in the same part, say, $\B'$.
Note that this implies that all sets in $\B_{r-2}$ that do not contain $x_r$ are in $\B'$.
Now fix any set $F\in \B_1$, let $F\cap A = \{x_i\}$ for some $i\in [r]$. Since $|A|=r$ and $|C|\ge r+1$, we can always pick a set in $\B_{r-2}$ that do not contain $x_r$ and is disjoint from $F$. So $F$ and thus all sets in $\B_1$ are in $\B'$.

Next assume $|A|=r-1$, say, $A=\{x_1,\dots, x_{r-1}\}$. Note that $|C|=m-(r-1)\ge r+2$.
Let $a=2$, $b=r-1$ and $s=r+1$. We replace $\{x_1,\dots, x_{r-2}\}$ by $z_1$ and set $z_2=x_{r-1}$. Applying Proposition~\ref{prop:1} shows that all sets in $\B_{r-2}$ that contain $\{x_1,\dots, x_{r-2}\}$ and all sets in $\B_1$ that contain $x_{r-1}$ are in the same part, say, $\B'$.
Note that if $r=3$, then we are done because $\B = \B_1$ and any set in $\B_1$ contains exactly one of $x_1$ and $x_2$.
Otherwise, $r\ge 4$. Note that all sets in $\B_{r-3}$ that do not contain $x_{r-1}$ are in $\B'$.
Now fix any set $F\in \B_1$, let $F\cap A = \{x_i\}$ for some $i\in [r-1]$. Since $|A|=r-1$ and $|C|\ge r+2$, we can always pick a set in $\B_{r-3}$ that do not contain $x_{r-1}$ and is disjoint from $F$. So $F$ and thus all sets in $\B_1$ are in $\B'$.
Thus the proof of the claim is complete.
\end{proof}

Now we show Lemma~\ref{lem:shift}. 
For a family $\C$ and an element $z$, let $\C(z) = \{E\cup \{z\}: E\in \C\}$.
Our scheme is as follows.
For the shift $S_{xy}: \h \rightarrow \F$,
let $\B_x$ be the subfamily of $\F$ such that
\[
\B_x = \{E\in \F: x\in E, y\notin E, (E\setminus \{x\})\cup \{y\}\notin \F \}
\]
and let $\B = \{E\setminus \{x\} : E\in \B_x\}$, which is a $(k-1)$-uniform family.
Clearly, only the sets in $\B_x$ might be obtained from the shift $S_{xy}$, i.e., $\F\setminus \B_x\subseteq \h$. So the shift $S_{xy}: \h \rightarrow \F$ can be interpreted as a partition $\B=\B_1\cup \B_2$ such that
\[
\h = (\F \setminus \B_x) \cup \B_1(y) \cup \B_2(x) \text{ and } \F = (\F \setminus \B_x) \cup \B_1(x)\cup \B_2(x),
\]
i.e., the sets in $\B_1(y)$ are shifted to $\B_1(x)$ by $S_{xy}$. Here
a natural requirement is that $\B_1(y) \cup \B_2(x)$ should be
intersecting, i.e., $(\B_1, \B_2)$ should be cross-intersecting.
We will show that $\B$ is non-separable, i.e., any cross-intersecting partition $\B=\B_1\cup \B_2$ satisfies that $\B=\B_1$ or $\B=\B_2$.
Observe that in all of our cases, $\B=\B_2$ means that $\h=\F$, and $\B=\B_1$ means that $\h$ is isomorphic to $\F$ with $y$ playing the role of $x$.
This will conclude the proof.

We mention that, in most cases, we will apply Proposition~\ref{prop:crossint} as follows. Let $Z = X\setminus \{x,y\}$ and $r=k-1$. Note that $|Z| \ge 2k+1 - 2 = 2r+1$. Our goal is to define $A$ appropriately so that we can apply Proposition~\ref{prop:crossint} and then conclude that $\B$ is non-separable.
We call the shift $S_{xy}: \h \rightarrow \F$ \emph{trivial} if $\B_x = \emptyset$.

\begin{proof}[Proof of Lemma~\ref{lem:shift}]
The proof consists of three cases on $\F$.

\medskip
\noindent\textbf{Case 1.} $\F = \G_{k-1}$. 
We may assume $k\ge 4$, since when $k=3$, $\G_{k-1}=\G_2$, which will be solved in Case 2.
We use the notation for $\G_{k-1}$ in Definition~\ref{ex:1}:
for a $(k-1)$-set $E\subset X$ and $x_0\in X\setminus E$, let $\G_{k-1}$ be the $k$-uniform family such that
\[
\G_{k-1} = \{G: E\subseteq G\} \cup \{ G: x_0\in G, G\cap E \neq \emptyset \}.
\]
The family~$\G_{k-1}$ partitions~$X$ into three types of elements:
\begin{itemize}
\item Type 1: $T_1=\{x_0\}$, 
\item Type 2: $T_2=E$ and 
\item Type 3: the set $T_3$ of the remaining elements ($|T_3|= n - k \ge k+1$).
\end{itemize}
Observe that shifts between two elements of the same type are trivial and for any $i<j$, any shift from an element of $T_i$ to an element of $T_j$ is trivial.
So we have the following two cases.

We first assume $x=x_0$.
Let $Z = X\setminus \{x,y\}$, $r=k-1\ge 3$ and $A=E\setminus \{y\}$. Note that $|A|\in \{r-1,r\}$ and we have $\B = \{B\subseteq Z: 0<|B\cap A| < |A|\}$.
So we can apply Proposition~\ref{prop:crossint} and the proof of this case is finished.

Next we assume that $x\in T_2$ and $y\in T_3$.
%\[
%\B = \{(E\cup \{z\})\setminus \{x\}: z\neq y, z\text{ is Type 3}\} \cup \{ G\setminus \{x\}: x_0\in G, G\cap E = \{x\}, y\notin G \}.
%\]
Let $C = T_3\setminus \{y\}$, we have
\[
\B = \{(E\setminus \{x\})\cup \{z\}: z\in C\} \cup \{ \{x_0\}\cup F: F\subseteq C, |F|=k-2 \}.
\]
%Note that any set of $\B$ contains either $E\setminus \{x\}$ or $x_0$. 
Let $a = 1$, $b=k-2$ and $s=k-1$. We replace $E\setminus \{x\}$ by $z_1$ and set $z_2=x_0$.
Applying Proposition~\ref{prop:1} concludes the proof.

\medskip
\noindent\textbf{Case 2.} $\F = \G_2$. 
Observe that $\G_2$ contains only two types of elements: the elements in the 3-set, and the other elements. Observe that the shifts between two elements of the same type are trivial. Also, if $y$ is in the 3-set, then the shift is trivial.
So assume that the 3-set is $\{x, x_1, x_2\}$ and $y\notin \{x, x_1, x_2\}$.
Let $C=X\setminus \{x, x_1, x_2, y\}$, $a=b=k-2$ and $s=2k-2$.
So we have
\[
\B = \{\{x_1\}\cup D: D\subseteq C, |D|=a\} \cup \{\{x_2\}\cup E: E\subseteq C, |E|=b\}
\]
and we conclude this case by Proposition~\ref{prop:1}.
%Let $A = \{x_1, x_2\}$ and note that $\B = \{B\subseteq Z: 0<|B\cap A| < |A|\}$.
%So we can apply Proposition~\ref{prop:crossint} and conclude the proof of this case.

\medskip
\noindent\textbf{Case 3.} $\F = \J_2$. 
We use the notation for $\J_2$ in Definition~\ref{ex:1}:
for a $(k-1)$-set $E\subseteq X$ and a $3$-set $J=\{x_0, x_1, x_2\}\subseteq X\setminus E$, let $\J_2$ be a $k$-uniform family such that
\[
\J_2 = \{G: E\subseteq G, G\cap J\neq\emptyset\}\cup \{G: J\subseteq G\}\cup \{ G: x_0\in G, G\cap E \neq \emptyset \}.
\]
The family~$\J_2$ partitions~$X$ into four types of elements:
\begin{itemize}
\item Type 1: $T_1=\{x_0\}$, 
\item Type 2: $T_2=E$, 
\item Type 3: $T_3=\{x_1, x_2\}$ and
\item Type 4: the set $T_4$ of the remaining elements ($|T_4|=n-k-2\ge k-1$).
\end{itemize}
Observe that shifts between two elements of the same type are trivial and for any $i<j$, any shift from an element of $T_i$ to another element of $T_j$ is trivial.

We first assume $x=x_0$. If $y\in T_3$, let $A=E$ and note that $\B = \{B\subseteq Z: 0<|B\cap A| < |A|\}$.
So we can apply Proposition~\ref{prop:crossint} and conclude the proof of this case.
Otherwise, $y\in T_2\cup T_4$. This case is more involved. 
Let $A=E\setminus \{y\}$ and thus $|A|\in \{r-1,r\}$. 
Consider a cross-intersecting partition $\B = \B_1\cup \B_2$.
Observe that $\{B\subseteq Z: 0< |B\cap A|< |A|\}\subseteq \B$.
We partition $\B$ into three subfamilies $\B^{*}, \B^0, \B^{**}$, where
\begin{align*}
&\B^* = \{B\subseteq Z: 0< |B\cap A|< |A|\}, \\
&\B^0 = \{B\in \B: B\cap A=\emptyset\} = \{\{x_1, x_2\}\cup F: F\subseteq T_4\setminus\{y\},\, |F|=k-3 \} \text{ and} \\
&\B^{**} = \{B\in \B: A\subseteq B\} = 
\begin{cases}
 \{A\cup\{z\}: z\in T_4 \} &\text{ if }y\in T_2, \\
 \{A\} &\text{ if } y\in T_4.
\end{cases}
\end{align*}

To get a cross-intersecting partition of $\B$, we need to distribute the sets in $\B^*\cup \B^0\cup \B^{**}$.
Note that $\B^*=\emptyset$ if and only if $|A|=1$, which, in turn, is equivalent to $k=3$ and $y\in T_2$.
In this case~$\B^0$ contains only $\{x_1, x_2\}$, which is disjoint from all other sets in $\B^{**}$ and thus $\B$ is non-separable.

Now suppose $\B^*\neq \emptyset$.
We apply Proposition~\ref{prop:crossint} and get that $\B^*$ is non-separable, i.e., $\B^* \subseteq \B_1$ or $\B^* \subseteq \B_2$. 
We first consider any set $P\in \B^0$. 
%Observe that $x_1, x_2\in F$ and $A\cap F=\emptyset$.
Fix any $a\in A$. Since $|Z\setminus \{a\}|\ge 2r$, there exists an $r$-set $P'\subseteq Z\setminus (\{a\}\cup P)$ such that $0< |P'\cap A|< |A|$ and $P\cap P'=\emptyset$. 
Thus, $P\in \B^0$ and $P'\in \B^*$ must be in the same part.
So all sets of $\B^*\cup \B^0$ belong to the same part.
Next, consider any set $B\in \B^{**}$ and note that $|B\cap T_4|\le 1$. Clearly, since $|T_4|\ge k-1$, there exists a $(k-3)$-set $F\subseteq T_4\setminus \{y\}$ such that $B\cap F=\emptyset$.
So $P=F\cup \{x_1, x_2\}\in B^0$ and $B\cap P=\emptyset$, which implies that $P$ and $B$ belong to the same part.
Thus, we conclude that $\B = \B_1$ or $\B = \B_2$ and we are done.

Next we assume $x\in T_2$. Let $E_i = (E\cup\{x_i\})\setminus \{x\}$ for $i=1,2$. Observe that if $y\in T_4$, then
\begin{align*}
\B = \{E_1, E_2\} \cup \{ G: x_0\in G, \,G\cap E = \emptyset,\, |G\cap \{x_1, x_2\}|\le 1, y\notin G \}.
\end{align*}
Otherwise $y\in T_3$, without loss of generality, $y=x_1$, then
\[
\B = \{E_2\}\cup \{ G: x_0\in G, \,G\cap E = \emptyset, \,G\cap \{x_1, x_2\}=\emptyset \}.
\]

In the former case, since $|T_4\setminus \{y\}|\ge k-2$, there is a set $B\in \B\setminus \{E_1, E_2\}$ such that $B\cap \{x_1, x_2\} = \emptyset$. 
Note that $B\cap E_1 = B\cap E_2 = \emptyset$, so $E_1, E_2$ must belong to the same part.
Moreover, for any set $B'\in \B\setminus \{E_1, E_2\}$, because $|B'\cap \{x_1, x_2\}|\le 1$, we have $B'\cap E_1=\emptyset$ or $B'\cap E_2=\emptyset$. Thus $\B$ is non-separable and we are done.
In the latter case, observe that for any $B\in \B\setminus \{E_2\}$, $E_2\cap B=\emptyset$.
So all sets in $\B$ must belong to the part in which $E_2$ is and we are done. 

At last, we assume that $x\in T_3$ and $y\in T_4$. Without loss of generality, let $x=x_1$.
In this case, we have that 
\[
\B = \{E\}\cup \{G: \{x_0, x_2\}\subseteq G, G\cap E=\emptyset, y\notin G\}.
\]
Clearly, for any $B\in \B\setminus \{E\}$, $E\cap B=\emptyset$.
So all sets in $\B$ must belong to the part in which $E$ is and we are done. 
\end{proof}

\subsection{The case $n\ge 8$ for $k=3$ and $n\ge 2k+1$ for $k\ge 4$}
We assume that the equality in \eqref{eq:eq0} or \eqref{eq:eq} holds.
We first show that $\G$, the family obtained by the shifts, is isomorphic to one of the extremal examples.
Indeed, to have equality we must have equality in Lemma~\ref{lem:calc}.
Note that $|\A_2| = k-1$ implies that $\A_2$ is a star of size $k-1$ or a triangle (for $k=4$ only).
If $k=4$ and $\A_2$ is a triangle at $\{x, y, z\}$, then any member of $\bigcup_{1\le i\le k}\A_i$ (and thus any member of $\G$) must contain at least two elements of $\{x, y, z\}$, which means that $\G\subseteq_S \G_2$.
Otherwise, suppose $\A_2$ is a star at $x$ of size $k-1$, say $\{xx_1, \dots, xx_{k-1}\}$.
Note that all sets in $\G$ not containing $x$ must contain $\{x_1,\dots, x_{k-1}\}$. Moreover, there are at least two such sets by Fact~\ref{fact:HM}.
So according to the number of such sets in $\G$, we have $\G\subseteq_S \G_{k-1}$ (only for $k\ge 4$) or $\G\subseteq_S \J_i$ for some $2\le i\le k-1$. 
Straightforward calculations show that the extremal value of $|\G|$ is achieved by $|\G_{3}|$ or $|\J_2|$ when $k=4$, and by $|\J_2|$ only when $k \neq 4$.
Then we are done by Lemma~\ref{lem:shift}.

\subsection{The case $n=7$ and $k=3$}

In this case we have to go through the shifting argument again. 
We assume that $\h$ is of the maximal size subject to the assumptions, i.e., $|\h|=12$.
Recall that if we apply~$S_{xy}$ repeatedly to~$\h$, then we obtain a family in one of the following four cases,
\begin{itemize}
\item[(0)] a family $\G$ which is {stable}, i.e., $S_{xy}(\G) = \G$ holds for all $x<y$,
\item[(1)] a family $\h_1$ such that $S_{xy}(\h_1)$ is EKR at $x$,
\item[(2)] a family $\h_2$ such that $\G'=S_{xy}(\h_2)$ is HM at $x$, or
\item[(3)] a family $\h_3$ such that $\G''=S_{xy}(\h_3)$ is HM at $\{x, x_1, x_2\}$ for some $x_1, x_2\in X\setminus \{x, y\}$.
\end{itemize}

%It remains to show the cases when we reach Case (2) or (3). 
We will use the following fact, the proof of which is straightforward and is
thus omitted.

\begin{fact}\label{fact:K23}
Let $\D_0$ be the graph $K_{2,3}$ and let $\D_1$ be the graph obtained from deleting any edge of $\D_0$. 
Let $\mathscr D$ be the set of graphs which are obtained from deleting any two edges of $K_5$. 
Then $\D_0$, $\D_1$ and all families in $\mathscr D$ are non-separable.
%let $\D_i=\D'\cup \D''$ be a partition such that $(\D', \D'')$ is cross-intersecting, then $\D_i=\D'$ or $\D_i= \D''$.
\end{fact}

First assume that we reach Case (2) and note that $|\G'|=12 = |\F_1|-1$. 
We use the following notation. Let $F$ be a $3$-set of $X$ and $x\in X\setminus F$. Let
\[
\F_1 := \{ F\}\cup \{G\subseteq X: x\in G, F\cap G\neq\emptyset\}.
\]
Note that if $y\in F$, then $\B$ is isomorphic to $\D_0$ or $\D_1$ and
if $y\notin F$, then $\B$ is isomorphic to some $\D_2\in \mathscr D$.
By Fact~\ref{fact:K23}, in either case, we know that
$\h_2\subseteq \F_1$, a contradiction.

Second assume that we reach Case (3) and note that $|\G''|=12 = |\G_2|-1$.
Since $\G''$ misses only one set of $\G_2$, we know $\{x, x_1, x_2\}\in \G''$ -- otherwise $\{y, x_1, x_2\}$ would have been shifted to $\{x, x_1, x_2\}$, a contradiction.
Observe that if $\{y, x_1, x_2\}\in \G''$, then $\B$ is isomorphic to $\D_0$ or $\D_1$.
By Fact~\ref{fact:K23}, in either case, we know that $\h_3\subseteq \G_2$, a contradiction.
Otherwise, $\{x_1, x_2\}\in \B$ and $\B\setminus \{x_1, x_2\}$ is isomorphic to $\D_0$.
So in this case $\B$ is not non-separable.
However, since $\D_0$ is non-separable, the only non-trivial partition of $\B$ is $\{\{x_1, x_2\}\}$ and $\B\setminus \{x_1, x_2\}$.
In all cases (with trivial partitions or the non-trivial partition of $\B$), it is easy to see that $\h_3\subseteq \G_2$, a contradiction.

Finally we assume that we reach Case (1). Let $X_1=\{x, y\}$.
We apply the shifts~$S_{x'y'}$ for $x',y'\in X\setminus X_1$ and $x'<y'$. 
Similar arguments in Proposition~\ref{prop:adjust} show that we will not reach Case (1) again and by the previous two cases, we will not reach any of Case (2) or (3).
%We continue the shift $S_{x'y'}$ for $x',y'\in X\setminus X_1$ and $x'<y'$ and we get a stable family $\G$ or we reach Case (2) or (3).
%We assume that the former case holds and deal with the other two cases later.
So we must get a family $\G$ such that $S_{x'y'}(\G)=\G$ for all $x', y'\in X \setminus X_1$, $x'<y'$.

Let $X_0=\emptyset$ and for $i=0,1$, let $Y$ be the union of $X_i$ and the first $6-|X_i|$ elements in $X\setminus X_i$.
Note that we have $|Y \setminus X_i| \ge 3=2k -3$ and thus Lemma~\ref{lem:2kint} holds. 
Then the proof of Lemma~\ref{lem:calc} gives that $|\A_1|=0$, $|\A_2|\le 2$ and $|\A_3|\le \binom52= 10$.
So we get $|\G|= |\A_2| + |\A_3|\le 12$.
To have equality we must have $|\A_2|=2$ and the same argument in Section 3.2 implies that $\G=\J_2$.
By Lemma~\ref{lem:shift}, $\h=\J_2$.

\section{Concluding remarks}
\label{sec:concluding-remarks}

We have restricted ourselves to intersecting families in this note,
and did not consider $t$-intersecting families (that is, families in which the intersection of any two of its members has at least~$t$ elements) for~$t>1$.  It would be
natural to investigate this more general setting.  Naturally, the
starting point would be the generalization of the Hilton--Milner
theorem to $t$-intersecting families~\cite{ahlswede96:_complete} (see
also~\cite{balogh08:_ahlsw_khach}).

One may also consider going beyond Theorem~\ref{thm:BeyHM}: that is,
\textit{what is the maximum size of an intersecting family $\h$ that
  is neither EKR, nor HM, nor is contained in~$\J_2$ $($and in~$\G_2$
  and~$\G_3$ if~$k=4$$)$?}
%\textit{Is~$\J_3$ the extremal family?}  
Very recently, Kostochka and
Mubayi~\cite{kostochka16:_structure} established that the answer is $|\J_3|$ for all large enough~$n$ \footnote{This was also independently observed by Jiang \cite{Jiang}.}.  
%\textit{Does this hold also when~$n$ and~$k$ are of comparable size?} 
Naturally, it would be good to know the complete answer, that is, for all $n$ and $k$ with $n> 2k$.

\section*{Acknowledgement}
We thank an anonymous referee for many helpful comments that helped us
improve the presentation of the paper. 
We thank Dhruv Mubayi for drawing our attention to~\cite{Frankl80} and for useful discussions.

\bibliographystyle{amsplain}
%\bibliography{Jan2014}
\bibliography{refs_JH}
\endgroup

\end{document}